\documentclass[12pt, english]{article}
\usepackage[T1]{fontenc}
\usepackage[latin9]{inputenc}
\usepackage{babel}

\usepackage{textcomp}
\usepackage{amsthm}
\usepackage{amsmath}
\usepackage{amssymb}
\usepackage{esint}
\usepackage[unicode=true, pdfusetitle,
 bookmarks=true,bookmarksnumbered=true,bookmarksopen=true,bookmarksopenlevel=1,
 breaklinks=false,pdfborder={0 0 1},backref=false,colorlinks=false]
 {hyperref}
\hypersetup{
 pdfstartview=FitH}

\theoremstyle{plain}
\newtheorem{thm}{Theorem}[section]
  \theoremstyle{definition}
  \newtheorem{defn}[thm]{Definition}
  \theoremstyle{remark}
  \newtheorem{rem}[thm]{Remark}
  \theoremstyle{plain}
  \newtheorem{lem}[thm]{Lemma}
 \theoremstyle{definition}
  \newtheorem{example}[thm]{Example}
  \theoremstyle{plain}
  \newtheorem{prop}[thm]{Proposition}
  \theoremstyle{plain}
  \newtheorem{cor}[thm]{Corollary}

\def\N{\mathbb{N}}
\def\No{{\mathbb{N}_0}}
\def\R{\mathbb{R}}
\def\Z{\mathbb{Z}}
\def\C{\mathbb{C}}
\def\Nn{{\mathbb{N}^n}}
\def\Non{{\mathbb{N}_0^n}}
\def\Rn{{\mathbb{R}^n}}
\def\Zn{{\mathbb{Z}^n}}
\def\Cn{\mathbb{C}^n}

\begin{document}

\baselineskip=17pt


\global\long\def\N{\mathbb{N}}
 \global\long\def\No{\mathbb{N}_{0}}
 \global\long\def\R{\mathbb{R}}
 \global\long\def\Z{\mathbb{Z}}
 \global\long\def\C{\mathbb{C}}
 \global\long\def\Nn{\mathbb{N}^{n}}
 \global\long\def\Non{\mathbb{N}_{0}^{n}}
 \global\long\def\Rn{\mathbb{R}^{n}}
 \global\long\def\Zn{\mathbb{Z}^{n}}
 \global\long\def\Cn{\mathbb{C}^{n}}
 \global\long\def\K{\mathbb{K}}

\title{On the Hausdorff dimension of continuous functions belonging to H\"older
and Besov spaces on fractal $d$-sets}

\author{Abel Carvalho\thanks{Centro I\&D Matemática e Aplicações, Universidade de Aveiro, 3810-193 Aveiro, Portugal, \texttt{abel.carvalho@ua.pt}}$\,$ and António Caetano\thanks{Departamento de Matemática, Universidade de Aveiro, 3810-193 Aveiro, Portugal, \texttt{acaetano@ua.pt} (corresponding author)}}
\date{}
\maketitle
\begin{abstract}
The Hausdorff dimension of the graphs of the functions in H\"older
and Besov spaces (in this case with integrability $p\geq1$) on fractal
$d$-sets is studied. Denoting by $s\in(0,1]$ the smoothness parameter, the
sharp upper bound $\min\{d+1-s,d/s\}$ is obtained. In particular, when passing from $d\geq s$ to $d<s$ there is a
change of behaviour from $d+1-s$ to $d/s$ which implies that even highly nonsmooth
functions defined on cubes in $\Rn$ have not so rough graphs when
restricted to, say, \emph{rarefied} fractals. 
\end{abstract}


\smallskip


\smallskip

\textbf{Acknowledgements:} Research partially supported by Funda\c
c\~ao para a Ci\^encia
e a Tecnologia (Portugal) through Centro de I\&D em Matemática e Aplicações (formerly Unidade de Investigação em Matem\'atica e Aplica\c c\~oes) of the University of Aveiro.

\section{Introduction}

This paper deals with the relationship between dimensions of sets
and of the graphs of real continuous functions defined on those sets
and having some prescribed smoothness. First studies in this direction
are reported in \cite[Chapter 10, § 7]{Kah93}, where $\min\{d+1-s,d/s\}$ is shown to
be an upper bound for the Hausdorff dimension of the graphs of H\"older continuous
functions with H\"older exponent $s\in(0,1)$ and defined on compact subsets of
$\R^n$ with Hausdorff dimension equal to $d$. 

That the above bound is sharp comes out from \cite[Chapter 18, § 7]{Kah93}

A corresponding result involving Besov spaces, on cubes on $\R^n$, with smoothness parameter
$s\in(0,1]$ and integrability parameter $p\in(0,\infty]$
was established by F. Roueff \cite[Theorem 4.8, p. 77]{Roueff-tese},
where it turned out that for $p\geq1$ the sharp upper bound is
$n+1-s$. On the other hand, if upper box dimension is used instead, then the complete picture was settled by A. Carvalho \cite{Abel05a}, with A. Deliu and B. Jawerth \cite{DeJa92} as forerunners (though the latter paper contains a mistake noticed by A. Kamont and B. Wolnik \cite{KaWo99} as well as, independently, by A. Carvalho \cite{Abel05a}).

Our aim here is to study the corresponding problem for Besov spaces when the underlying
domains for the functions are allowed to have themselves non-integer dimensions (as was the case in the mentioned results involving H\"older spaces).
More precisely, we consider $d$-sets in $\Rn$ (with $0<d\leq n)$
for our underlying domains and determine the sharp upper bound for
the Hausdorff dimension of the graphs of continuous functions defined
on such $d$-sets and belonging to Besov spaces (with integrability parameter $p\geq1$), with a prescribed smoothness parameter $s\in(0,1]$.
As in the case of H\"older spaces, under the assumption $s\leq d$ we obtain the behaviour
$d+1-s$, whereas when $s>d$ the correct sharp upper bound is $d/s$. 

One of the qualitative implications of this change of behaviour for
small values of $d$ is the following (we illustrate it in the case of H\"older continuous functions): 

Given $s\in(0,1]$, $n\in\N$ and a positive integer $d\in(0,n]$,
it is possible to find an H\"older continuous functions defined on
a cube in $\Rn$ and with H\"older exponent $s$ whose restriction
to some $d$-set has graph with roughness (as measured by the Hausdorff
dimension) as close to $d+1-s$ as one wishes. In particular, if we
start with an $s$ close to zero, our restriction to a $d$-set might
give us a function with a graph having dimension close to $d+1$.
This is also true for $d\in(0,1)$ as long as $d\geq s$, so in such
cases the graphs of our functions, even when these are restricted
to $d$-sets with small $d$, might gain almost one extra unit of
roughness when compared with the domain. Consequently we might be
near the right endpoint of the interval obtained in Lemma \ref{lem:minmax}.
However, when $d$ is allowed to become less than $s$, the dimension
of the corresponding graphs cannot overcome $d/s$, so letting $d$
tend to zero will result in graphs with dimensions also approaching
zero. In other words, in such cases we are definitely near the left
endpoint of the interval obtained in Lemma \ref{lem:minmax}.

So we show that the same type of phenomenon occurs in the setting of Besov spaces.
The proof of the upper bound $d+1-s$ is inspired in the deep proof
given by Roueff in the case of having $n$-cubes for domains. On the other hand, the proof of
the sharpness owes a lot to the ideas used by Hunt \cite{Hunt98}, where a combination between randomness and the potential theoretic method for the estimation of Hausdorff dimensions has been used.

For the sake of completeness and to help having the more complex results in perspective, we revisit also the simpler setting of H\"older spaces and give shorter proofs than in the setting of Besov spaces.
For the reader only interested in the result involving H\"older continuous
functions, some material can be skipped: Lemma \ref{lem:aggregation};
everything after Example \ref{exa:holder} within subsection \ref{sub:Function-spaces};
Theorem \ref{thm:Besov up est} (and its long proof, of course); everything
after the first paragraph in the proof of Corollary \ref{cor:Corollary}.
This material has specifically to do with the proof involving Besov
spaces.

\section{Preliminaries}

In this section we give the necessary definitions concerning the dimensions,
sets and function spaces to be considered. We also recall some results
and establish others that will be needed for the main proofs in the
following section. However, we start by listing some notation which
applies everywhere in this paper:

The number $n$ is always considered in $\N$. The closed ball in
$\Rn$ with center $a$ and radius $r$ is denoted by $B_{r}(a)$
and a cartesian product of $n$ intervals of equal length is said
to be an $n$-cube. The notation $|\cdot|$ stands either for the
Euclidean norm in $\Rn$ or for the sum of coordinates of a multi-index
in $\Non$ and $\lambda_{n}$ denotes the Lebesgue measure in $\R^{n}$. 

The shorthand ${\rm diam}$ is used for the diameter of a set and
${\rm osc}_{I}f$ stands for the oscillation of the function $f$
on the set $I$ (that is, the difference $\sup_{I}f-\inf_{I}f$),
whereas $\Gamma(f)$ denotes the graph of the function $f$. The usual
Schwartz space of functions on $\Rn$ is denoted by $\mathcal{S}(\R^{n})$,
its dual $\mathcal{S}'(\R^{n})$ being the usual space of tempered
distributions.

On the relation side, $a\lesssim b$ (or $b\gtrsim a$) applies to
nonnegative quantities $a$ and $b$ and means that there exists a
positive constant $c$ such that $a\leq cb$, whereas $a\approx b$
means that $a\lesssim b$ and $b\lesssim a$ both hold. On the other
hand, $A\subset B$ applies to sets $A$ and $B$ and is the usual
inclusion relation (allowing also for the equality of sets). We use
$A\hookrightarrow B$ when continuity of the embedding is also meant,
for the topologies considered in the sets $A$ and $B$.

\subsection{Dimensions, sets and functions}

We start with the dimensions, after recalling some notions related
with measures. These definitions and results are taken from \cite{Fal90},
to which we refer for details.
\begin{defn}
(a) A \emph{measure} on $\R^{n}$ is a function $\mu:\mathcal{P}(\R^{n})\to[0,\infty]$,
defined over all subsets of $\R^{n}$, which satisfies the following
conditions: (i) $\mu(\emptyset)=0$; (ii) $\mu(U_{1})\leq\mu(U_{2})$
if $U_{1}\subset U_{2}$; (iii) $\mu(\cup_{k\in\N}U_{k})\leq\sum_{k\in\N}\mu(U_{k})$,
with equality in the case when $\{U_{k}:\, k\in\N\}$ is a collection
of pairwise disjoint Borel sets.

(b) A \emph{mass distribution} on $\R^{n}$ is a measure $\mu$ on
$\R^{n}$ such that $0<\mu(\R^{n})<\infty$.

(c) The \emph{support} of a measure $\mu$ is the smallest closed
set $A$ such that $\mu(\R^{n}\setminus A)=0$.
\end{defn}
\vspace{0mm}

\begin{defn}
Let $d\geq0$, $\delta>0$ and $\emptyset\not=E\subset\R^{n}$.

(a) We define $\mathcal{H}_{\delta}^{d}(E):=\inf\{\sum_{k\in\N}{\rm diam}\,(U_{k})^{d}:\,{\rm diam}\,(U_{k})\leq\delta\,\mbox{ and }\, E\subset\cup_{k\in\N}U_{k}\}$.

(b) The quantity $\mathcal{H}_{\delta}^{d}(E)$ increases when $\delta$
decreases. Hence the following definition, of the so-called $d$-dimensional
\emph{Hausdorff measure}, makes sense: $\mathcal{H}^{d}(E):=\lim_{\delta\to0^{+}}\mathcal{H}_{\delta}^{d}(E)$.
And it is, indeed, a measure according to the preceding definition.

(c) There exists a critical value $d_{E}\geq0$ such that $\mathcal{H}^{d}(E)=\infty$
for $d<d_{E}$ and $\mathcal{H}^{d}(E)=0$ for $d>d_{E}$. We define
the \emph{Hausdorff dimension} of $E$ as $\dim_{H}E\,:=d_{E}$.\end{defn}
\begin{rem}
In order to get to the definiton of Hausdorff dimension of the set
$E$ we can restrict consideration to sets $U_{k}$ which are $n$-cubes,
with sides parallel to the axes, of side length $2^{-j}$, with $j\in\N,$
and centered at points of the type $2^{-j}m$, with $m\in\Z^{n}$.
This can lead to different values for measures, but will produce the
same Hausdorff dimension as in the definition above. We shall take
advantage of this later on.
\end{rem}
We collect in the following remark some properties concerning Hausdorff
dimension that we shall also need:
\begin{rem}
\label{rem:4 prop Hausd dim}(a) $\dim_{H}\R^{n}=n$. More generally,
the same is true for the dimension of any open subset of $\R^{n}$.

(b) If $E\subset F$, then $\dim_{H}E\leq\dim_{H}F$.

(c) The (Hausdorff) dimension does not increase under a Lipschitzian
transformation of sets.

(d) Consider a closed subset $E$ of $\R^{n}$and a mass distribution
$\mu$ supported on $E$. Let $t>0$ be such that$\int_{\R^{n}}\int_{\R^{n}}\frac{1}{|x-y|^{t}}\: d\mu(x)\, d\mu(y)<\infty.$
Then $\dim_{H}E\geq t$.\end{rem}
\begin{defn}
\label{def:upper box dim}Let $E$ be a non-empty bounded subset of
$\R^{n}$. The \emph{upper box counting dimension} of $E$ is the
number\[
\overline{\dim}_{B}E:=\limsup_{j\to\infty}\frac{\log_{2}N_{j}(E)}{j},\]
where $N_{j}(E)$ stands for the number of $n$-cubes, with sides
parallel to the axes, of side length $2^{-j}$ and centered at points
of the type $2^{-j}m$, with $m\in\Z^{n}$, which intersect $E$.
\end{defn}
Again, we collect in a remark some properties which will be of
use later on:
\begin{rem}
\label{rem:2 prop dim}(a) Let $E$ be a non-empty bounded subset
of $\R^{n}$. Then $\dim_{H}E\leq\overline{\dim}_{B}E$.

(b) Let $\emptyset\not=E\subset\R^{n}$, $\emptyset\not=F\subset\R^{m}$,
with $E$ bounded. Then $\dim_{H}(E\times F)\leq\overline{\dim}_{B}E+\dim_{H}F$.

Next we define the sets which we want to consider:\end{rem}
\begin{defn}
\label{def:d-set} Let $0<d\leq n$. A $d$-set in $\R^{n}$ is a
(compact) subset $K$ of $\R^{n}$ which is the support of a mass
distribution $\mu$ on $\R^{n}$ satisfying the following condition:\[
\exists c_{1},c_{2}>0:\,\forall r\in(0,1],\,\forall x\in K,\; c_{1}r^{d}\leq\mu(B_{r}(x))\leq c_{2}r^{d}.\]
\end{defn}
\begin{rem}
\label{rem:dim d}Here we are not following \cite{Fal90}, but rather
\cite{JW84}, just with the difference that our $d$-sets are necessarily
compact, because we are assuming that our associated measure $\mu$
is actually a mass distribution. Otherwise we can follow \cite{JW84}
and conclude that we can take for $\mu$ the restriction to $K$ of
the $d$-dimensional Hausdorff measure and that a $d$-set has always
Hausdorff dimension equal to $d$.
\end{rem}
\vspace{0mm}

\begin{rem}
\label{rem:dsetcover} Any $d$-set $K$ in $\Rn$, with $d\in(0,n]$,
intersects $\approx r^{d}$ cubes of any given regular tessellation
of $\Rn$ by cubes of sides parallel to the axes and side length $r^{-1}$,
for any given $r\geq r_{0}>0$, $r_{0}$ fixed, with equivalence constants
independent of $r$. For a proof, adapt to our setting the arguments
in \cite[Lemma 2.1.12]{Mou01b}.
\end{rem}
As we shall be interested, later on, to study dimensions of graphs
of functions, we consider here a couple of results involving these
special sets. We start with a result which gives already some restrictions
for the possible values that the Hausdorff dimension of such sets
can have.
\begin{lem}
\label{lem:minmax}If $f$ is a real function defined on a $d$-set
$K$, then $\dim_{H}\Gamma(f)\in[d,d+1]$.\end{lem}
\begin{proof}
(i) We prove first that $\dim_{H}\Gamma(f)\geq d$. Writing the elements
of $\Gamma(f)$ in the form $(x,t)$, with $x\in K\subset\R^{n}$
and $t=f(x)\in\R$, it is clear that $K=P\Gamma(f)$, where $P:\R^{n+1}\to\R^{n}$
is the projection defined by $P(x,t):=x$. As is easily seen, $P$
is a Lipschitzian transformation of sets, therefore, by Remarks \ref{rem:dim d}
and \ref{rem:4 prop Hausd dim}(c), $d=\dim_{H}K\leq\dim_{H}\Gamma(f)$.

(d) In order to show that, on the other hand, $\dim_{H}\Gamma(f)\leq d+1$,
just notice that Remarks \ref{rem:4 prop Hausd dim}(a),(b), \ref{rem:2 prop dim}(b),
\ref{rem:dsetcover} and Definition \ref{def:upper box dim} allow
us to write that \[
\dim_{H}\Gamma(f)\leq\dim_{H}K\times\R\leq\overline{\dim}_{B}K+1=d+1.\]

\end{proof}
Next we state and prove a \textit{seed} for the so-called\textit{
aggregation method} considered in \cite[p. 24, discussion after Remark 2.2]{Roueff-tese}:
\begin{lem}
\label{lem:aggregation}Let $k,l\in\No$, with $k<l$, and $h_{0}$,
$h_{1}$ be two bounded real functions defined on a bounded subset
$K$ of $\R^{n}$. Let $\mathcal{Q}_{j}$, with $j=k,l$, be finite
coverings of $K$ by $n$-cubes $Q_{j}$ with sides parallel to the
axes, of side length $2^{-j}$ and centered at points of the type
$2^{-j}m$, with $m\in\Z^{n}$. Assume that a covering of $\Gamma(h_{0})$
by $(n+1)$-cubes of side length at least $2^{-k}$ is given, and
such that \emph{over} each $Q_{k}$ each point between the \emph{levels}
$m_{Q_{k}}:=\inf_{Q_{k}\cap K}h_{0}$ and $M_{Q_{k}}:=\sup_{Q_{k}\cap K}h_{0}$
belongs to one of those $(n+1)$-cubes. Then the number of $(n+1)$-cubes
of side length $2^{-l}$ that one needs to add to the given covering
of $\Gamma(h_{0})$, in order to get a covering of $\Gamma(h_{0}+h_{1})$
by $(n+1)$-cubes of side length at least $2^{-l}$, and such that
\emph{over} each $Q_{l}$ each point between the \emph{levels} $m_{Q_{l}}:=\inf_{Q_{l}\cap K}h_{0}+h_{1}$
and $M_{Q_{l}}:=\sup_{Q_{l}\cap K}h_{0}+h_{1}$ belongs to one of
those $(n+1)$-cubes, is bounded above by\[
\sum_{Q_{l}\in\mathcal{Q}_{l}}(2^{l+1}\sup_{y\in Q_{l}}|h_{1}(y)|+2).\]
\end{lem}
\begin{proof}
Clearly what one needs is to cover the portion of $\Gamma(h_{0}+h_{1})$
\emph{over}\textit{\emph{ each $Q_{l}$ and then put all together.
Notice that, given one such $Q_{l}$, what one really needs to cover
is the portion of $\Gamma(h_{0}+h_{1})$ }}\textit{over}\textit{\emph{
$Q_{k}\cap Q_{l}$, for the $Q_{k}$containing $Q_{l}$. By hypothesis,
all the points of $\R^{n+1}$}}\textit{over}\textit{\emph{ $Q_{k}\cap Q_{l}$
between the }}\textit{levels}\textit{\emph{ $m_{Q_{k}}$ and $M_{Q_{k}}$
are already covered, so one only needs to ascertain which points of
$\Gamma(h_{0}+h_{1})$ over $Q_{k}\cap Q_{l}\,(=Q_{l})$ do not fall
within those }}\textit{levels}\textit{\emph{. Since \[
m_{Q_{k}}-\sup_{y\in Q_{l}}(-h_{1}(y))=m_{Q_{k}}+\inf_{y\in Q_{l}}h_{1}(y)\leq h_{0}(x)+h_{1}(x)\leq M_{Q_{k}}+\sup_{y\in Q_{l}}h_{1}(y)\]
for all $x\in Q_{l}\cap K$, then it is enough to add, }}\textit{over}\textit{\emph{
$Q_{l}$, to the previous cover, a number of $(n+1)$-cubes of side
length $2^{-l}$ in a quantity not exceeding $2\,(2^{l}\sup_{y\in Q_{l}}|h_{1}(y)|+1)$.}}
\end{proof}
We shall also need a couple of technical lemmas which we state and
prove next:
\begin{lem}
\label{lem:preliminar}Let $\beta>\alpha>0$ be two fixed numbers
and $\zeta:(0,\beta)\to\R^{+}$ be a fixed integrable function. Then
\[
\int_{0}^{\alpha}\zeta(r)f(r)\, dr\approx\int_{0}^{\beta}\zeta(r)f(r)\, dr\]
for all non-increasing functions $f:(0,\beta)\to\R^{+}$.\end{lem}
\begin{proof}
$\int_{\alpha}^{\beta}\zeta(r)f(r)\, dr\leq\int_{\alpha}^{\beta}\zeta(r)f(\alpha)\, dr=c\int_{0}^{\alpha}\zeta(r)f(\alpha)\, dr\leq\int_{0}^{\alpha}\zeta(r)f(r)\, dr$.\end{proof}
\begin{lem}
Consider $0<d\leq n$ and $K$ a $d$-set. Assume that $\mu_{K}$
is a mass distribution supported on $K$ according to Definition \ref{def:d-set}.
Then\[
\int_{\Rn}f(|x-y|)\, d\mu_{K}(x)\approx\int_{0}^{{\rm diam}\, K}r^{d-1}f(r)\, dr\]
for all $y\in K$ and all non-increasing and continuous functions
$f:\R^{+}\to\R^{+}$.\end{lem}
\begin{proof}
Let $R\in(0,\infty)$ be such that $K\subset B_{R}(y)$, with $R$
chosen independently of $y$ (this is possible, due to the boundedness
of $K$). By the definition of $d$-set, there exists $c_{1},c_{2}>0$
(also independent of $y$) such that \[
c_{1}r^{d}\leq\mu_{K}(B_{r}(y))\leq c_{2}r^{d},\qquad\mbox{for all }r\in(0,R].\]

For any $\lambda\in(0,R)$, consider the annulus \[
C^{(\lambda)}:=\{x\in\Rn:\; R-\lambda<|x-y|\leq R\}\]
and, for any $N\in\N$, the partition $\{C_{l}:\, l=1,\ldots,N\}$
of this annulus, where\[
C_{l}:=\{x\in\Rn:\, R-\frac{l\lambda}{N}<|x-y|\leq R-\frac{(l-1)\lambda}{N}\},\qquad\mbox{for }l=1,\ldots,N.\]
We have the following estimations:\begin{eqnarray*}
\int_{C^{(\lambda)}}f(|x-y|)\, d\mu_{K}(x) & \leq & \sum_{l=1}^{N}\mu_{K}(C_{l})f(R-\frac{l\lambda}{N})\\
 & = & \varepsilon_{N}+\sum_{l=1}^{N}\mu_{K}(C_{l})f(R-\frac{(l-1)\lambda}{N})\\
 & \leq & \varepsilon_{N}+\sum_{l=1}^{N}(c_{2}r_{l-1}^{d}-c_{2}r_{l}^{d})f(r_{l-1})\\
 & \leq & \varepsilon_{N}+\int_{(c_{1}/c_{2})^{1/d}(R-\lambda)}^{R}\frac{d(c_{2}r^{d})}{dr}f(r)\, dr\\
 & = & \varepsilon_{N}+c_{2}d\int_{(c_{1}/c_{2})^{1/d}(R-\lambda)}^{R}r^{d-1}f(r)\, dr\end{eqnarray*}
 where $r_{l}:=c_{2}^{-1/d}\mu_{K}(B_{R-l\lambda/N}(y))^{1/d}$, so
$r_{l}\leq R-\frac{l\lambda}{N}$, $l=0,\ldots,N$. Furthermore, $\varepsilon_{N}$
equals\begin{eqnarray*}
\lefteqn{\sum_{l=1}^{N}\mu_{K}(C_{l})(f(R-\frac{l\lambda}{N})-f(R-\frac{(l-1)\lambda}{N}))}\\
 & \leq & \mu_{K}(C^{(\lambda)})\max_{l=1,\ldots,N}|f(R-\frac{l\lambda}{N})-f(R-\frac{(l-1)\lambda}{N})|,\end{eqnarray*}
from which follows (using the uniform continuity of $f$ in $[R-\lambda,R]$)
that $\lim_{N\to\infty}\varepsilon_{N}=0$. In this way we obtain\begin{equation}
\int_{C^{(\lambda)}}f(|x-y|)\, d\mu_{K}(x)\leq c_{2}d\int_{(c_{1}/c_{2})^{1/d}(R-\lambda)}^{R}r^{d-1}f(r)\, dr.\label{eq:c2d}\end{equation}
And, by similar calculations,\begin{equation}
\int_{C^{(\lambda)}}f(|x-y|)\, d\mu_{K}(x)\geq c_{1}d\int_{(c_{2}/c_{1})^{1/d}(R-\lambda)}^{R}r^{d-1}f(r)\, dr,\label{eq:c1d}\end{equation}
as long as we only consider values of $\lambda\in(0,R)$ close enough
to $R$, so that $(c_{2}/c_{1})^{1/d}(R-\lambda)<R$. Letting now
$\lambda$ tend to $R$ in \eqref{eq:c2d} and \eqref{eq:c1d}, we
get (using also the fact that $\mu_{K}(\{0\})=0)$\[
c_{1}d\int_{0}^{R}r^{d-1}f(r)\, dr\leq\int_{\Rn}f(|x-y|)\, d\mu_{K}(x)\leq c_{2}d\int_{0}^{R}r^{d-1}f(r)\, dr.\]

The required result now follows by applying Lemma \ref{lem:preliminar}.\end{proof}
\begin{example}
\label{exa:par=0000EAnteses do lem0.2} As an interesting application
of the preceding result, which will be useful to us later on, we have\[
\int_{K}\frac{1}{|x-y|^{u}}\, d\mu_{K}(x)\approx\int_{0}^{{\rm diam}\, K}r^{d-u-1}\, dr,\]
with equivalence constants independent of $y\in K$ and $u\geq0$.
\end{example}

\subsection{Function spaces\label{sub:Function-spaces}}
\begin{defn}
For $s\in(0,1]$ and $\emptyset\not=K\subset\Rn$, define $\mathcal{C}^{s}(K):=\{f:K\to\R\,:\;\exists c>0:\forall x,y\in K,\,|f(x)-f(y)|\leq c\,|x-y|^{s}\}$.
In particular, all functions in $\mathcal{C}^{s}(K)$ are continuous
(of course, considering in $K$ the metric inherited from the sorrounding
$\Rn$). We shall call $\mathcal{C}^{s}(K)$ the set of the (real)
H\"older continuous functions (over $K$) of exponent $s$. On the
other hand, given $r\in\N$, we denote by $C^{r}(\Rn)$ the set of
all complex-valued functions defined on $\Rn$ such that the function
itself and all its derivatives up to (and including) the order $r$
are bounded and uniformly continuous.
\end{defn}
The following is a non-trivial example of a function in $\mathcal{C}^{s}(K)$.
Its proof follows from an easy adaptation of a corresponding result
in \cite[p. 796]{Hunt98}.
\begin{example}
\label{exa:holder}$x\mapsto W_{s,\theta}(x):=\sum_{i=1}^{n}\sum_{j=0}^{\infty}\rho^{-js}\cos(\rho^{j}x_{i}+\theta_{ij})$
is H\"older continuous of exponent $s\in(0,1)$ on any bounded subset
of $\R^{n}$, with H\"older constant independent of $\theta$.
\end{example}
The following definition (of Daubechies wavelets) includes an existence
assertion. For details, we refer to \cite[section 3.1]{Tri06}.
\begin{defn}
\label{def:wavelets} Let $r\in\N$. Define $L_{0}:=1$ and $L:=L_{j}:=2^{n}-1$
if $j\in\N$. There exist compactly supported real functions $\psi_{0}\in C^{r}(\Rn)$
and $\psi^{l}\in C^{r}(\Rn)$, $l=1,\ldots,L$\begin{equation}
\mbox{( with }\,\int_{\Rn}x^{\alpha}\psi^{l}(x)\, dx=0,\ \alpha\in\Non,\ |\alpha|\leq r\mbox{ ),}\label{eq:moments}\end{equation}
such that $\{\Psi_{jm}^{l}:\, j\in\No,\,1\leq l\leq L_{j},\, m\in\Zn\}$
is an orthonormal basis in $L_{2}(\Rn)$, where, by definition,\[
\Psi_{jm}^{l}(x):=\begin{cases}
\psi_{0}(x-m) & \mbox{if }\, j=0,\, l=1,\, m\in\Zn\\
2^{\frac{j-1}{2}n}\psi^{l}(2^{j-1}x-m) & \mbox{if }\, j\in\N,\,1\leq l\leq L,\, m\in\Zn\end{cases}.\]

\end{defn}
The Besov spaces in the following definition are the usual ones (up
to equivalent quasi-norms), defined by Fourier-analytical tools (a
definition along this line can be seen in \cite[section 1.3]{Tri06},
for example). From this point of view the definition which follows
is actually a theorem: for details, see \cite[Theorem 3.5 and footnote in p. 156]{Tri06}.
\begin{defn}
\label{def:Besov-spaces} Let $0<p,q\leq\infty$, $s\in\R$ and $r$
be a natural number such that $r>\max\{s,n(1/p-1)_{+}-s\}$. The Besov
space $B_{pq}^{s}(\Rn)$ is the set of all sums \begin{equation}
f:=\sum_{j,l,m}\lambda_{jm}^{l}2^{-jn/2}\Psi_{jm}^{l}=\sum_{j=0}^{\infty}\sum_{l=1}^{L_{j}}\sum_{m\in\Zn}\lambda_{jm}^{l}2^{-jn/2}\Psi_{jm}^{l}\label{eq:series}\end{equation}
(convergence --- actually, unconditional convergence --- in $\mathcal{S}'(\Rn)$),
for all given sequences $\{\lambda_{jm}^{l}\in\C:\, j\in\No,\, l=1,\ldots,L_{j},\, m\in\Zn\}$
such that\begin{equation}
\left(\sum_{m\in\Zn}|\lambda_{0m}^{1}|^{p}\right)^{1/p}+\sum_{l=1}^{L}\left(\sum_{j=1}^{\infty}2^{j(s-n/p)q}\left(\sum_{m\in\Zn}|\lambda_{jm}^{l}|^{p}\right)^{q/p}\right)^{1/q}\label{eq:quasi-norm}\end{equation}
(with the usual modifications if $p=\infty$ or $q=\infty$) is finite.
It turns out that \eqref{eq:quasi-norm} defines a quasi-norm in $B_{pq}^{s}(\Rn)$
which makes this a complete space.\end{defn}
\begin{rem}
\label{rem:unif conv series}(a) The representation \eqref{eq:series}
is unique, that is, it is uniquely determined by the limit $f$, namely
the coefficients are determined by the formul\ae \[
\lambda_{jm}^{l}=2^{jn/2}(f,\Psi_{jm}^{l}),\quad j\in\No,\, l=1,\ldots,L_{j},\, m\in\Zn,\]
where $(\cdot,\cdot)$, though standing for the inner product in $L_{2}(\Rn)$
when applied to functions in such a space, must in general be understood
in the sense of the dual pairing $\mathcal{S}(\Rn)-\mathcal{S}'(\Rn)$
--- see \cite[section 3.1]{Tri06} for details. As a consequence,
when $f$ is compactly supported, then, given any $j\in\No$, only
finitely many coefficients $\lambda_{jm}^{l}$ are non-zero. 

(b) Arguing as in \cite[p. 21]{Roueff-tese}, we can say that the
convergence in \eqref{eq:series} is even uniform in the support of
$f$ whenever this support is compact and $f$ is a continuous function.\end{rem}
\begin{defn}
Consider $0<d\leq n$ and $K$ a $d$-set with associated mass distribution
$\mu$ according to Definition \ref{def:d-set}. For $0<p<\infty,$
we define the Lebesgue space $L_{p}(K)$ as the set of all $\mu$-measurable
functions $f:K\to\C$ for which the quasi-norm given by \[
\|f\|_{L_{p}(K)}:=\left(\int_{\Rn}|f(x)|^{p}\, d\mu(x)\right)^{1/p}\]
is finite.
\end{defn}
\vspace{0mm}

\begin{defn}
\label{def:trace}Consider $0<d\leq n$ and $K$ a $d$-set. Let $0<p,q<\infty$.
Assuming that there exists $c>0$ such that \begin{equation}
\|\varphi|_{K}\|_{L_{p}(K)}\leq c\,\|\varphi\|_{B_{pq}^{s}(\R^{n})},\qquad\varphi\in\mathcal{S}(\R^{n}),\label{eq:phitrace}\end{equation}
 the trace of $f\in B_{pq}^{s}(\R^{n})$ on $K$ is defined by $tr_{K}\, f\,:=\lim_{j\to\infty}\varphi_{j}|_{K}$
in $L_{p}(K)$, where $(\varphi_{j})_{j}\subset\mathcal{S}(\R^{n})$
is any sequence converging to $f$ in $B_{pq}^{s}(\R^{n})$.
\end{defn}
This definition is justified by the completeness of $L_{p}(K)$ and
by the fact that the restrictions on $p,q$ guarantee that the Schwartz
space is dense in the Besov spaces under consideration. That the definition
does not depend on the particular approaching sequence $(\varphi_{j})_{j}$
is a consequence of \eqref{eq:phitrace}.

By \cite[Theorem 18.6 and Comment 18.7]{Tri02}, which holds for $d=n$
too (cf. also \cite[Theorem 3.3.1(i)]{Bri-tese}), one knows that
the assumption \eqref{eq:phitrace} holds true when $0<p<\infty$,
$0<q\leq\min\{1,p\}$ and $s=\frac{n-d}{p}$. Therefore the trace
of functions of Besov spaces on $K$ is well-defined for that range
of parameters.

Since $B_{pq}^{s}(\R^{n})\hookrightarrow B_{p,\min{1,p}}^{\frac{n-d}{p}}(\R^{n})$
whenever $s>\frac{n-d}{p}$, then the trace as defined above makes
sense for functions of the spaces $B_{pq}^{s}(\R^{n})$, for any $0<p<\infty$,
$0<q<\infty$ and $s>\frac{n-d}{p}$. Moreover, since the embedding
between the Besov spaces above also hold when $q=\infty$, then, though
Definition \eqref{def:trace} can no longer be applied, we define,
for any $f\in B_{p\infty}^{s}(\R^{n})$, with $0<p<\infty$ and $s>\frac{n-d}{p}$,
the $tr_{K}\, f$ by its trace when $f$ is viewed as an element of
$B_{p,\min{1,p}}^{\frac{n-d}{p}}(\R^{n})$.

Finally, in the case $f\in B_{\infty q}^{s}(\R^{n})$, with $0<q\leq\infty$
and $s>0$, we define $tr_{K}\, f\,:=f|_{K}$, the pointwise restriction,
since, for such range of parameters, the elements of those Besov spaces
are all (represented by) continuous functions (actually, we even have
$B_{\infty q}^{s}(\R^{n})\hookrightarrow C(\R^{n})$,  where the
latter space stands for the set of all complex-valued, bounded and
uniformly continuous functions on $\R^{n}$ endowed with the $\sup$
norm).
\begin{rem}
Another way of defining trace on $K$ is by starting to define it
by pointwise restriction when the function is continuous and, at least
in the case when $f$ is locally integrable on $\R^{n}$, define its
trace on $K$ by the pointwise restriction $\overline{f}|_{K}$, where
\[
\overline{f}(x)\,:=\lim_{r\to0}\frac{1}{\lambda_{n}(B_{r}(x))}\int_{B_{r}(x)}f(y)\, dy\]
for the values of $x$ where the limit exists. It is known (Lebesgue
differentiation theorem) that $\overline{f}=f$ a.e. (for locally
integrable functions $f$). As is easily seen, the identity $\overline{f}(x)=f$(x)
surely holds at any point $x$ where $f$ is continuous.

This is the approach followed by Jonsson and Wallin in \cite{JW84}
(see pp. 14-15), where they have shown (it is a particular case of
\cite[Theorem 2 in p. 142]{JW84}) that, for $0<d\leq n$, $1\leq p,q\leq\infty$ and
$s>\frac{n-d}{p}$, the map $f\mapsto\overline{f}|_{K}$ takes $B_{pq}^{s}(\R^{n})$
linearly and boundedly into $L_{p}(K)$.

Note now that, for this restriction of parameters, both $tr_{K}\, f$
and $\overline{f}|_{K}$ coincide with $f|_{K}$ when $f\in\mathcal{S}(\R^{n})$.
Therefore, at least when we further restrict $p$ and $q$ to be finite,
we get the identity $tr_{K}\, f=\overline{f}|_{K}$ in $L_{p}(K)$
for any $f\in B_{pq}^{s}(\R^{n})$, by a density argument. In the
case we still restrict $p$ to be finite but admit $q=\infty$, from
our definition above we see that $tr_{K}\, f$ is also the trace of
$f\in B_{p,\min{1,p}}^{\varepsilon+\frac{n-d}{p}}(\R^{n})$, for any
suitable small $\varepsilon>0$, where here the parameter ``$q$''
is again finite, so also $tr_{K}\, f=\overline{f}|_{K}$ in $L_{p}(K)$.
This identity even holds when $p=\infty$ is admitted, taking into
account that in that situation we are dealing with continuous functions.

Summing up, when $0<d\leq n$, $1\leq p,q\leq\infty$ and $s>\frac{n-d}{p}$
we have $tr_{K}\, f\,=\overline{f}|_{K}$.
\end{rem}
Although in \cite{JW84} both $p$ and $q$ are assumed to be greater
than or equal to 1, we can proceed with our comparative analysis between
$tr_{K}\, f$ and $\overline{f}|_{K}$ even for the remaining positive
values of $q$. In fact, given $0<d\leq n$, $1\leq p\leq\infty$,
$0<q<1$ and $s>\frac{n-d}{p}$, and due to the embedding $B_{pq}^{s}(\R^{n})\hookrightarrow B_{p,\min{1,p}}^{\varepsilon+\frac{n-d}{p}}(\R^{n})$
(which holds for any suitable small $\varepsilon>0$), we see, from
what was mentioned before, that the trace of $f\in B_{pq}^{s}(\R^{n})$
can be seen as the trace of $f$ as an element of $B_{p,\min{1,p}}^{\varepsilon+\frac{n-d}{p}}(\R^{n})$;
since in the latter space the ``$q$'' parameter is in the range $[1,\infty]$
(actually, it is 1), then we already know that $tr_{K}\, f\,=\overline{f}|_{K}$
here too.

As a consequence we have also the following remark, which will be
useful later on:
\begin{rem}
\label{rem:continuoustrace}Consider $0<d\leq n$ and $K$ a $d$-set.
If $f$ is a continuous function belonging to $B_{pq}^{s}(\R^{n})$,
with $1\leq p\leq\infty$, $0<q\leq\infty$ and $s>\frac{n-d}{p}$,
then $tr_{K}\, f=f|_{K}$.\end{rem}
\begin{defn}
\label{def:Besov on K}Consider $0<d\leq n$ and $K$ a $d$-set.
Let $0<p,q\leq\infty$ and $s>0$. We define the Besov space $\mathbb{B}_{pq}^{s}(K)$
as the set of traces of the elements of $B_{pq}^{s+\frac{n-d}{p}}(\R^{n})$
endowed with the quasi-norm defined by\[
\|f\|_{\mathbb{B}_{pq}^{s}(K)}:=\inf\|g\|_{B_{pq}^{s+\frac{n-d}{p}}(\R^{n})}\]
where the infimum runs over all $g\in B_{pq}^{s+\frac{n-d}{p}}(\R^{n})$
such that $tr_{K}\, g=f$.
\end{defn}
For a motivation for such definition, see \cite[sections 20.2 and 20.3]{Tri97}.
From the considerations above it follows that the \textit{trace} maps
$B_{pq}^{s+\frac{n-d}{p}}(\R^{n})$ linearly and boundedly both into
$L_{p}(K)$ and $\mathbb{B}_{pq}^{s}(K)$, where the parameters are
as in the preceding definition.
\begin{prop}
\label{pro:emb p2<p1}Consider $0<d\leq n$ and $K$ a $d$-set. Let
$0<p_{2}<p_{1}\leq\infty$, $0<q\leq\infty$ and $s>0$. Then\[
\mathbb{B}_{p_{1}q}^{s}(K)\hookrightarrow\mathbb{B}_{p_{2}q}^{s}(K).\]

\end{prop}
A sketch of a proof for this result, at least for $d=n$, can be seen
in \cite[Step 2 in p. 165]{Tri97}. The argument is not clear to us
when $0<d<n$, but a proof in this situation can be seen in \cite[Proposition 2.18]{Mou01a},
where quarkonial decompositions were used. In both cases, a proof
with atomic decompositions can also be used instead.

From Remark \ref{rem:continuoustrace} it follows that the trace on
a $d$-set $K$, with $0<d\leq n$, of a continuous function belonging
to $B_{pq}^{s+\frac{n-d}{p}}(\R^{n})$, with $1\leq p\leq\infty$,
$0<q\leq\infty$ and $s>0$, is still a continuous function (on $K$).
In the sequel we shall need a partial converse for this result, the
proof of which is sketched below:
\begin{prop}
\label{pro:cont ext}Consider $0<d\leq n$ and $K$ a $d$-set. Let
$1\leq p,q\leq\infty$ and $0<s<1$. Any continuous function in $\mathbb{B}_{pq}^{s}(K)$
can be obtained as the trace (or pointwise restriction) of a continuous
function in $B_{pq}^{s+\frac{n-d}{p}}(\R^{n})$.\end{prop}
\begin{proof}
(i) We start with the case $0<d<n$. 

We shall need to consider a Whitney decomposition of $K^{c}$ by a
family of closed $n$-cubes $Q_{i}$ and an associated partition of
unity by functions $\phi_{i}$. We use here the notations and conventions
of \cite[pp. 23-24 and 155-157]{JW84}, except that our $K$ here
is in the place of $F$ over there. In particular, $x_{i}$, $s_{i}$
and $l_{i}$ shall, respectively, stand for the center of $Q_{i}$,
its side length and its diameter. Moreover, given $f\in\mathbb{B}_{pq}^{s}(K)$,
the function $\mathcal{E}f$ is defined by\[
\mathcal{E}f(x):=\sum_{i\in I}\phi_{i}(x)\frac{1}{\mu_{K}(B_{6l_{i}}(x_{i}))}\int_{|t-x_{i}|\leq6l_{i}}f(t)\, d\mu_{K}(t),\qquad x\in K^{c},\]
where $I$ is the set of indices $i$ such that $s_{i}\leq1$ and
$\mu_{K}$ is the mass distribution supported on $K$ according to
Definition \ref{def:d-set}. Notice that $\mathcal{E}f$ is defined
a.e. in $\R^{n},$ because the asumption $d<n$ guarantees that $K$
has Lebesgue measure 0.

According to \cite[Theorem 3 in p. 155]{JW84}, $\mathcal{E}f\in B_{pq}^{s+\frac{n-d}{p}}(\R^{n})$,
$\mathcal{E}f$ is $C^{\infty}$ in $K^{c}$ (so, in particular, it
is continuous on $K^{c}$) and $tr_{K}\mathcal{E}f=(\overline{\mathcal{E}f})|_{K}=f$.
Since $\overline{\mathcal{E}f}=\mathcal{E}f$ a.e., $\overline{\mathcal{E}f}$
is a representative of an element of $B_{pq}^{s+\frac{n-d}{p}}(\R^{n})$
whose pointwise restriction to $K$ coincides with $f$. Since the
identity $\overline{\mathcal{E}f}(x)=\mathcal{E}f(x)$ holds for any
$x\in K^{c}$, where we already know this function is continuous,
it remains to show that $\overline{\mathcal{E}f}$ is continuous on
$K$, i.e., that\[
\forall t_{0}\in K,\,\forall\varepsilon>0,\,\exists\delta>0:\,\forall x\in\R^{n},\, x\in B_{\delta}(t_{0})\Rightarrow|\overline{\mathcal{E}f}(x)-f(t_{0})|<\varepsilon.\]
The implication being trivially true when $x$ is also in $K$, we
can assume that $x\in B_{\delta}(t_{0})\cap K^{c}$, in which case
we have to arrive to the conclusion that $|\mathcal{E}f(x)-f(t_{0})|<\varepsilon$.

Pick one $Q_{k}$ containing $x$ and consider $t_{k}\in K$ and $y_{k}\in Q_{k}$
such that $|y_{k}-t_{k}|={\rm dist}\{Q_{k},K\}$. Notice that $|t_{k}-t_{0}|\leq3\delta$
and that\[
|\mathcal{E}f(x)-f(t_{0})|\leq|\mathcal{E}f(x)-f(t_{k})|+|f(t_{k})-f(t_{0})|,\]
hence, by the continuity of $f$ on $K$, for sufficiently small $\delta>0$
one gets $|f(t_{k})-f(t_{0})|<\varepsilon/2$ and, therefore, we only
need to show that $|\mathcal{E}f(x)-f(t_{k})|<\varepsilon/2$ too.

Assume that we will choose $\delta>0$ small enough, so that, in particular,
the restriction $s_{i}\leq1$ for the indices $i\in I$ is not really
a restriction and, therefore, $\sum_{i\in I}\phi_{i}(x)=1$. Then
\begin{eqnarray*}
|\mathcal{E}f(x)-f(t_{k})| & = & |\sum_{i\in I}\phi_{i}(x)\frac{1}{\mu_{K}(B_{6l_{i}}(x_{i}))}\int_{|t-x_{i}|\leq6l_{i}}f(t)\, d\mu_{K}(t)-\sum_{i\in I}\phi_{i}(x)f(t_{k})|\\
 & \leq & \sum_{i\in I}\phi_{i}(x)\frac{1}{\mu_{K}(B_{6l_{i}}(x_{i}))}\int_{|t-x_{i}|\leq6l_{i}}|f(t)-f(t_{k})|\, d\mu_{K}(t)\\
 & \lesssim & \frac{1}{l_{k}^{d}}\int_{|t-x_{k}|\leq27l_{k}}|f(t)-f(t_{k})|\, d\mu_{K}(t)\\
 & \leq & \frac{1}{l_{k}^{d}}\int_{|t-t_{k}|\leq32l_{k}}|f(t)-f(t_{k})|\, d\mu_{K}(t)\\
 & \lesssim & \max_{|t-t_{k}|\leq32l_{k}}|f(t)-f(t_{k})|.\end{eqnarray*}
The desired estimate follows then from the given continuity of $f$
on (the compact set) $K$, by choosing $l_{k}$ small enough, to which
it suffices to choose a small enough $\delta>0$.

(ii) Now we consider the case $d=n$.

Given a continuous $f\in\mathbb{B}_{pq}^{s}(K)$, we want to show
that there exists a continuous $g\in B_{pq}^{s+\frac{n-n}{p}}(\R^{n})$
such that $g|_{K}=f$. Define $f_{1}:\, K\times\{0\}\subset\R^{n+1}\to\C$
by $f_{1}(x,0):=f(x)$, obtaining in this way a continuous function
in $\mathbb{B}_{pq}^{s}(K\times\{0\})$. Since $K\times\{0\}$ is
an $n$-set in $\R^{n+1}$, with $0<n<n+1$, we can apply part (i)
to say that there exists a continuous $g_{1}\in B_{pq}^{s+\frac{1}{p}}(\R^{n+1})$
with $g_{1}|_{K\times\{0\}}=f_{1}$, and from here one gets that $g:\,\R^{n}\to\C$
given by $g(x):=g_{1}(x,0)$ is the required function. We have taken
advantage of an old trace result, which can, for example, be seen
in \cite[Theorem 3 in p. 19]{JW84}, which states that $B_{pq}^{s}(\R^{n})$
can be identified with the traces on $\R^{n}\times\{0\}$ of the elements
of $B_{pq}^{s+\frac{1}{p}}(\R^{n+1})$.
\end{proof}

\section{Main results and proofs}

With $0\leq d\leq n$ and $0<s\leq1$, define \begin{equation}
H(d,s):=\begin{cases}
d+1-s & \mbox{ if }s<d\\
d/s & \mbox{ if }s\geq d\end{cases},\label{eq:H(d,s)}\end{equation}

\noindent or, what turns out to be the same (cf. also the end of the proof of the next proposition), $H(d,s):=\min\{d+1-s,d/s\}$.

\begin{prop}
\label{pro:lessthan}Let $0<s\leq1$ and $K$ be a $d$-set in $\Rn$,
with $0<d\leq n$. If $f\in\mathcal{C}^{s}(K)$ then $\dim_{H}\Gamma(f)\leq H(d,s)$.\end{prop}
\begin{proof}
We use the inequality $\dim_{H}\Gamma(f)\leq\overline{\dim}_{B}\Gamma(f)$
and estimate the latter dimension.

As mentioned in Remark \ref{rem:dsetcover}, given any $j\in\N$,
$K$ can be covered by $c_{1}\,2^{jd}$ cubes of side length $2^{-j}$
in a corresponding regular tessellation of $\Rn$ by dyadic cubes
of sides parallel to the axes. The part of the graph of $f$ \emph{over}
any one of such cubes can, obviously, be covered by $c_{2}\,2^{-j(s-1)}$
cubes of side length $2^{-j}$ of a regular tessellation of $\R^{n+1}$
by corresponding dyadic cubes of sides parallel to the axes. Therefore
$\overline{\dim}_{B}\Gamma(f)\leq\limsup_{j\to\infty}\frac{\log_{2}(c_{1}c_{2})+j(d-s+1)}{j}=d+1-s$.

Alternatively, and using again Remark \ref{rem:dsetcover}, given
any $j\in\N$, $K$ can be covered by $c_{1}\,2^{jd/s}$ cubes of
side length $2^{-j/s}$ in a corresponding regular tessellation of
$\Rn$ by cubes of sides parallel to the axes. The part of the graph
of $f$ \emph{over} any one of such cubes can, obviously, be covered
by $c_{2}$ parallelepipeds of height $2^{-j}$, so that the whole
graph can be covered by $c_{1}\, c_{2}\,2^{jd/s}$ of such parallelepipeds.
Since each one of these is covered by at most $2^{n+1}$ cubes of
side length $2^{-j}$ of a regular tessellation of $\R^{n+1}$ by
corresponding dyadic cubes of sides parallel to the axes, then we
see that $c_{3}\,2^{jd/s}$ of the latter cubes are enough to cover
$\Gamma(f)$, hence $\overline{\dim}_{B}\Gamma(f)\leq\limsup_{j\to\infty}\frac{\log_{2}(c_{3})+jd/s}{j}=d/s$.

Observe now that, apart from the obvious case $s=1$, the inequality
$d/s\leq d+1-s$ holds if, and only if, $d\leq s$, which concludes
the proof.
\end{proof}
Let $[0,2\pi]$ be endowed with its uniform Lebesgue measure, so that
it becomes a probability space. In what follows, $\Pi$ shall stand
for the product space $([0,2\pi]^{\N})^{n}$ of $n$ copies of the
infinite product space $[0,2\pi]^{\N}$. The elements of $\Pi$ shall
usually be denoted by $\theta$ and we shall commit the abuse of notation
of denoting by $d\theta$ both the measure in $\Pi$ and integration
with respect to such measure, as in $\int_{\Pi}\, d\theta$.
\begin{thm}
\label{thm:weierstrass}Let $\rho>1$, $0<s<1$, $\theta=((\theta_{ij})_{j\in\N})_{i=1,\ldots,n}\in\Pi$
and $W_{s,\theta}$ be the function defined by\[
W_{s,\theta}(x):=\sum_{i=1}^{n}\sum_{j=0}^{\infty}\rho^{-js}\cos(\rho^{j}x_{i}+\theta_{ij}),\qquad x=(x_{1},\ldots,x_{n})\in\Rn.\]
Let $0<d\leq n$ and $K$ be a $d$-set. Then\[
\dim_{H}\Gamma(W_{s,\theta}|_{K})=H(d,s)\quad\theta\mbox{-a.e.},\]
where $H(d,s)$ is as defined in \eqref{eq:H(d,s)}.\end{thm}
\begin{proof}
Due to Proposition \ref{pro:lessthan} and Example \ref{exa:holder},
the inequality $\leq$ is clear, even for all $\theta$.

In order to prove the opposite inequality, we use the criteria of
Remark \ref{rem:4 prop Hausd dim}(d).

Let $\mu_{\theta}$ be the Borel measure supported on $\Gamma(W_{s,\theta}|_{K})$
defined by $\mu_{K}\circ(I,W_{s,\theta})^{-1}$, where $I$ is the
identity in $\Rn$ and $\mu_{K}$ is the mass distribution supported
on $K$ according to Definition \ref{def:d-set}.

Given $t>0$,\begin{eqnarray}
\lefteqn{\lefteqn{\int_{\Pi}\int_{\R^{n+1}}\int_{\R^{n+1}}\frac{1}{|P-Q|^{t}}\, d\mu_{\theta}(P)\, d\mu_{\theta}(Q)\, d\theta}}\nonumber \\
 & = & \int_{\Pi}\int_{\R^{n}}\int_{\R^{n}}\frac{1}{|(I,W_{s,\theta})(x)-(I,W_{s,\theta})(y)|^{t}}\, d\mu_{K}(x)\, d\mu_{K}(y)\, d\theta\nonumber \\
 & = & \int_{\Rn}\int_{\Rn}\int_{\Pi}\frac{1}{(|x-y|^{2}+(W_{s,\theta}(x)-W_{s,\theta}(y))^{2})^{t/2}}\, d\theta\, d\mu_{K}(x)\, d\mu_{K}(y)\nonumber \\
 & = & \int_{\Rn}\int_{\Rn}\int_{\R}\frac{1}{(|x-y|^{2}+z^{2})^{t/2}}\,(d\theta\circ A^{-1})(z)\, d\mu_{K}(x)\, d\mu_{K}(y),\label{eq:tripleint}\end{eqnarray}
 where, for each fixed $x,y\in\Rn$, $A:\Pi\to\R$ is the function
defined by $A(\theta):=W_{s,\theta}(x)-W_{s,\theta}(y)$. 

Now observe that $A(\theta)=\sum_{i=1}^{n}A_{i}(\theta)$, with \[
A_{i}(\theta)=\sum_{j=0}^{\infty}q_{ij}\sin(r_{ij}+\theta_{ij}),\]
where $q_{ij}$ and $r_{ij}$ do not depend on $\theta$. Adapting
\cite[pp. 797-798]{Hunt98} to our setting, under the assumption $0<|x_{i}-y_{i}|<\frac{\pi}{\rho^{2}}$,
the measure $d\theta\circ A^{-1}$ is absolutely continuous with respect
to the Lebesgue measure in $\R$, with density function $h_{i}$ satisfying
the estimate\begin{equation}
h_{i}(z)\leq C\,|x_{i}-y_{i}|^{-s},\label{eq:majhi}\end{equation}
 where the positive constant $C$ depends only on $\rho$. It is also
easily seen that\[
d\theta\circ(A_{i},\sum_{{{\scriptstyle k=1}\atop {\scriptstyle k\not=i}}}^{n}A_{k})^{-1}=(d\theta\circ A_{i}^{-1})\otimes(d\theta\circ(\sum_{{{\scriptstyle k=1}\atop {\scriptstyle k\not=i}}}^{n}A_{k})^{-1}),\]
hence the density function $h$ of $d\theta\circ A^{-1}$ is given
by the convolution of the density functions $h_{i}$ of $d\theta\circ A_{i}^{-1}$
and, say, $\breve{h_{i}}$ of $d\theta\circ(\sum_{{{\scriptstyle k=1}\atop {\scriptstyle k\not=i}}}^{n}A_{k})^{-1}$.
Fixing now an $i$ such that $|x_{i}-y_{i}|=\max_{1\leq k\leq n}|x_{k}-y_{k}|$
and assuming that $0<|x-y|<\frac{\pi}{\rho^{2}}$, from \eqref{eq:majhi}
we then get\begin{equation}
h(z)=(h_{i}\star\breve{h_{i}})(z)\leq(\sup_{w\in\R}h_{i}(w))\int_{\R}\breve{h_{i}}(t)\, dt\leq C\, n^{s/2}\,|x-y|^{-s},\qquad\forall z\in\R,\label{eq:majh}\end{equation}
 where $C$ is the same constant as in \eqref{eq:majhi}.

Returning to \eqref{eq:tripleint}, we can now write, taking into
account that the hypotheses guarantee that $(\mu_{K}\otimes\mu_{K})(\{x=y\})=0$,
\begin{eqnarray}
\lefteqn{\int_{\Pi}\int_{\R^{n+1}}\int_{\R^{n+1}}\frac{1}{|P-Q|^{t}}\, d\mu_{\theta}(P)\, d\mu_{\theta}(Q)\, d\theta}\label{eq:tripleint-1}\\
 & = & \int_{{{\scriptstyle \Rn\times\Rn}\atop {\scriptstyle |x-y|\geq\frac{\pi}{\rho^{2}}}}}\int_{\R}\frac{h(z)}{(|x-y|^{2}+z^{2})^{t/2}}\, dz\, d(\mu_{K}\otimes\mu_{K})(x,y)\nonumber \\
 &  & +\;\int_{{{\scriptstyle \Rn\times\Rn}\atop {\scriptstyle 0<|x-y|<\frac{\pi}{\rho^{2}}}}}\int_{\R}\frac{h(z)}{(|x-y|^{2}+z^{2})^{t/2}}\, dz\, d(\mu_{K}\otimes\mu_{K})(x,y),\nonumber \end{eqnarray}
where the first term on the right-hand side is clearly finite, while,
due to \eqref{eq:majh}, the second term can be estimated from above
by\begin{eqnarray}
\lefteqn{C\, n^{s/2}\,\int_{{{\scriptstyle \Rn\times\Rn}\atop {\scriptstyle 0<|x-y|<\frac{\pi}{\rho^{2}}}}}\int_{\R}\frac{|x-y|^{-s}}{(|x-y|^{2}+z^{2})^{t/2}}\, dz\, d(\mu_{K}\otimes\mu_{K})(x,y)}\nonumber \\
 & \!\!\!\!\!\!\!\!\!\!\!\!\!\!\!\! = & \!\!\!\!\!\!\!\!\!\! C\, n^{s/2}\,\int_{{{\scriptstyle \Rn\times\Rn}\atop {\scriptstyle 0<|x-y|<\frac{\pi}{\rho^{2}}}}}|x-y|^{-s+1-t}\int_{\R}\frac{1}{(1+w^{2})^{t/2}}\, dw\, d(\mu_{K}\otimes\mu_{K})(x,y).\label{eq:bothcases}\end{eqnarray}

We now need to split the proof in two cases, in order to proceed.

\medskip{}

\textbf{Case $d>s$:} 

Consider \[
t_{m}:=d+1-s-\frac{1}{m},\]
for sufficiently large $m\in\N$ so that $d-s>\frac{2}{m}$. Then,
using $t_{m}$ in the place ot $t$, the inner integral in \eqref{eq:bothcases}
can be estimated from above by\[
\int_{|w|\leq1}\, dw+\int_{|w|>1}\frac{1}{|w|^{t_{m}}}\, dw\;\leq\;2+\frac{4}{d-s},\]
hence \eqref{eq:bothcases} can be estimated from above by\begin{eqnarray*}
\lefteqn{c_{1}\,\int_{{{\scriptstyle \Rn\times\Rn}\atop {\scriptstyle 0<|x-y|<\frac{\pi}{\rho^{2}}}}}|x-y|^{-s+1-t_{m}}\, d(\mu_{K}\otimes\mu_{K})(x,y)}\\
 & \leq & c_{1}\,\int_{\Rn}\int_{\Rn}|x-y|^{-d+1/m}\, d\mu_{K}(x)\, d\mu_{K}(y)\\
 & \leq & c_{2}\,\int_{K}\int_{0}^{{\rm diam}\, K}r^{1/m-1}\, dr\, d\mu_{K}(y)\;<\;\infty,\end{eqnarray*}
where we have used Example \ref{exa:par=0000EAnteses do lem0.2}.
Therefore, we have proved the finiteness of \eqref{eq:tripleint-1}
when using $t=t_{m}=d+1-s-\frac{1}{m}$ for any sufficiently large
natural $m$, and have shown in particular, for any such number $t_{m}$,
that\[
\int_{\R^{n+1}}\int_{\R^{n+1}}\frac{1}{|P-Q|^{t_{m}}}\, d\mu_{\theta}(P)\, d\mu_{\theta}(Q)<\infty\qquad\theta\mbox{-a.e.}.\]
Consequently, by Remark \ref{rem:4 prop Hausd dim}(d),\[
\dim_{H}(\Gamma(W_{s,\theta}|_{K})\geq t_{m}=d+1-s-\frac{1}{m}\qquad\theta\mbox{-a.e.}\]
for any $m$ large enough. Since this is a countable number of possibilities,
we can also state that, for almost all $\theta\in\Pi$, $\dim_{H}(\Gamma(W_{s,\theta}|_{K})\geq d+1-s-\frac{1}{m}$
for all previously considered numbers $m$, so that the required result
follows after letting $m$ tend to infinity.

\medskip{}

\textbf{Case $d\leq s$:}

Here we shall take advantage of the fact that the support of $d\theta\circ A^{-1}$
is contained in $[-c\,|x-y|^{s},c\,|x-y|^{s}]$, for some positive
constant $c$ (independent of $x$ and $y$). That this is the case
follows from Example \ref{exa:holder}.

We return then to the decomposition given above for \eqref{eq:tripleint-1}
and observe that we can replace the integral over $\R$ in the second
term in that decomposition by a corresponding integral over $[-c\,|x-y|^{s},c\,|x-y|^{s}]$,
so that instead of \eqref{eq:bothcases} we can write, up to a constant
factor,\begin{equation}
\int_{{{\scriptstyle \Rn\times\Rn}\atop {\scriptstyle 0<|x-y|<\frac{\pi}{\rho^{2}}}}}|x-y|^{-s+1-t}\int_{0}^{c\,|x-y|^{s-1}}\frac{1}{(1+w^{2})^{t/2}}\, dw\, d(\mu_{K}\otimes\mu_{K})(x,y),\label{eq:bothcases-1}\end{equation}
where, moreover, $c\,|x-y|^{s-1}$ can, without loss of generality,
be assumed to be greater than 1.

Consider now\[
t_{m}:=\frac{d}{s}-\frac{1}{m},\]
for sufficiently large $m\in\N$ so that $t_{m}>0$. Then, using $t_{m}$
in the place ot $t$, the inner integral in \eqref{eq:bothcases-1}
can be estimated from above by\[
\int_{0}^{1}\frac{1}{(1+w^{2})^{t_{m}/2}}\, dw+\int_{1}^{c\,|x-y|^{s-1}}\frac{1}{(1+w^{2})^{t_{m}/2}}\, dw\;<\;\frac{c^{1-t_{m}}}{1-t_{m}}\,|x-y|^{(1-t_{m})(s-1)},\]
hence \eqref{eq:bothcases-1} can be estimated from above by\begin{eqnarray*}
\lefteqn{\frac{c_{3}}{1-t_{m}}\,\int_{{{\scriptstyle \Rn\times\Rn}\atop {\scriptstyle 0<|x-y|<\frac{\pi}{\rho^{2}}}}}|x-y|^{-s+1-t_{m}+(1-t_{m})(s-1)}\, d(\mu_{K}\otimes\mu_{K})(x,y)}\\
 & \leq & \frac{c_{3}}{1-t_{m}}\,\int_{\Rn}\int_{\Rn}|x-y|^{-d+s/m}\, d\mu_{K}(x)\, d\mu_{K}(y)\\
 & \leq & \frac{c_{4}}{1-t_{m}}\int_{K}\int_{0}^{{\rm diam}\, K}r^{s/m-1}\, dr\, d\mu_{K}(y)\;<\;\infty,\end{eqnarray*}
where we have used Example \ref{exa:par=0000EAnteses do lem0.2}.
The rest of the proof follows as in the previous case, the difference
being that now $t_{m}=\frac{d}{s}-\frac{1}{m}$, therefore tends to
$\frac{d}{s}$ when $m$ goes to infinity. \end{proof}
\begin{thm}
\label{thm:Besov up est}Consider $0<d\leq n$ and $K$ a $d$-set.
Let $1\leq p\leq\infty$, $0<q\leq\infty$ and $0<s\leq1$. Let $f$
be any real continuous function in $\mathbb{B}_{pq}^{s}(K)$. Then
$\dim_{H}\Gamma(f)\leq H(d,s)$.\end{thm}
\begin{proof}
We deal first with the case $0<s<1$.

We start by remarking that $\mathbb{B}_{pq}^{s}(K)\hookrightarrow\mathbb{B}_{1q}^{s}(K)\hookrightarrow\mathbb{B}_{1\infty}^{s}(K)$.
The first of these embeddings comes from Proposition \ref{pro:emb p2<p1};
the second one is a direct consequence of a well-known corresponding
embedding for Besov spaces on $\R^{n}$. Since $H(d,s)$ does not
depend on $p$ nor $q$, it is then enough to prove our Theorem for
the Besov spaces $\mathbb{B}_{1\infty}^{s}(K)$.

Given any real continuous function $f\in\mathbb{B}_{1\infty}^{s}(K)$,
let $g\in B_{1\infty}^{s+n-d}(\R^{n})$ be a continuous extension
of $f$ (there exists one, by Proposition \ref{pro:cont ext}). Because
$K$ is bounded, we can, without loss of generality, also assume that
$g$ is compactly supported (if necessary, we can always multiply
it by a suitable cut-off function). By Definition \ref{def:Besov-spaces},
we can write\begin{equation}
g=\sum_{j\in\No}\sum_{l=1}^{L_{j}}\sum_{m\in\Z^{n}}\lambda_{jm}^{l}2^{-jn/2}\Psi_{jm}^{l},\quad\mbox{ unconditional convergence in }\mathcal{S}'(\R^{n}),\label{eq:unc conv g}\end{equation}
where \begin{equation}
\sup_{{j\in\No\atop l=1,\ldots,L_{j}}}2^{j(s-d)}(\sum_{m\in\Z^{n}}|\lambda_{jm}^{l}|\,)<\infty.\label{eq:est sum lambda}\end{equation}
Notice also, in view of Remark \ref{rem:unif conv series}, that the
convergence in \eqref{eq:unc conv g} is also uniform and that, given
each $j\in\No$, only a finite number of coefficients $\lambda_{jm}^{l}$
are different from zero. 

Denoting \[
h:=\sum_{j\in\N}\sum_{l=1}^{L_{j}}\sum_{m\in\Z^{n}}\lambda_{jm}^{l}2^{-jn/2}\Psi_{jm}^{l}=\sum_{j\in\N}\sum_{l=1}^{L}\sum_{m\in\Zn}\lambda_{jm}^{l}2^{-n/2}\psi^{l}(2^{j-1}\cdot-m)\]
(which, clearly, is also a continuous function belonging to $B_{1\infty}^{s+n-d}(\R^{n})$),
we can write\[
f=\sum_{m\in\Z^{n}}\lambda_{0m}^{l}\Psi_{0m}^{l}|_{K}+h|_{K}.\]
Since the sum on $m$ is a Lipschitz function (as we have remarked
above, this sum actually has only a finite number of non-zero terms),
then $\Gamma(f)$ and $\Gamma(h|_{K})$ have the same Hausdorff dimension
(this follows from Remark \ref{rem:4 prop Hausd dim}(c)). Therefore,
our proof will be finished if we show that $\dim_{H}\Gamma(h|_{K})\leq H(d,s)$.
This is what we are going to prove next, assuming, for ease of writing,
that our $h$ is simply given by\begin{equation}
h:=\sum_{j\in\N}\sum_{m\in\Z^{n}}\lambda_{jm}\psi(2^{j-1}\cdot-m)=\sum_{j\in\N}\sum_{m\in\Zn}\lambda_{jm}\psi_{jm}.\label{eq:h}\end{equation}
So, we got rid of the finite summation in $l$ and of the unimportant
factor $2^{-n/2}$ and simplified the notation for the $\lambda$'s
and $\psi$'s (introducing also the simplification $\psi_{jm}:=\psi(2^{j-1}\cdot-m)$).
In this way we keep the essential features of the method without unecessarily
overcrowding the notation. If we were to consider the exact form of
$h$ in what follows, after some point we could indeed get rid of
the finite summation in $l$ without changing the estimates that are
obtained up to multiplicative positive constants.

In order to estimate $\dim_{H}\Gamma(h|_{K})$ from above, we are
going to estimate, also from above, the quantities $\mathcal{H}_{\sqrt{n+1}\,2^{-j_{1}+1}}^{t}(\Gamma(h|_{K}))$,
for $t\geq0$ and $j_{1}\in\N\setminus\{1\}$. 

We start by estimating\[
\mathcal{H}_{\sqrt{n+1}\,2^{-j_{1}+1}}^{t}(\Gamma(h_{0}|_{K}+\sum_{j=j_{1}}^{j_{2}-1}h_{j}|_{K}),\]
where $j_{2}>j_{1}$ (with $j_{2}\in\N$),\[
h_{0}:=\sum_{j=1}^{j_{1}-1}\sum_{m\in\Z^{n}}\lambda_{jm}\psi_{jm}\quad\mbox{ and }\quad h_{j}:=\sum_{m\in\Z^{n}}\lambda_{jm}\psi_{jm},\quad j=j_{1},\ldots,j_{2}-1,\]
by applying Lemma \ref{lem:aggregation} a finite number of times.
In what follows we use the notations $Q_{j}$ and $\mathcal{Q}_{j}$
with the same meaning as in that lemma (assuming further that the
coverings $\mathcal{Q}_{j}$ are minimal) and each time we start from
a covering of\[
\Gamma(h_{0}|_{K}+\sum_{j=j_{1}}^{k}h_{j}|_{K}),\quad k=j_{1}-1,j_{1},\ldots,j_{2}-2\]
(with the understanding that when $k=j_{1}-1$ we are starting from
a covering of $\Gamma(h_{0}|_{K})$) by $(n+1)$-cubes of side length
at least $2^{-k}$ and such that \emph{over} each $Q_{k}$ each point
between the \emph{levels} $m_{Q_{k}}:=\inf_{Q_{k}\cap K}(h_{0}|_{K}+\sum_{j=j_{1}}^{k}h_{j}|_{K})$
and $M_{Q_{k}}:=\sup_{Q_{k}\cap K}(h_{0}|_{K}+\sum_{j=j_{1}}^{k}h_{j}|_{K})$
belongs to one of those $(n+1)$-cubes. Therefore, by applying Lemma
\ref{lem:aggregation}, each time we conclude that the number of $(n+1)$-cubes
of side length $2^{-(k+1)}$ that one needs to add to the previous
covering, in order to get a covering of $\Gamma(h_{0}|_{K}+\sum_{j=j_{1}}^{k+1}h_{j}|_{K})$
by $(n+1)$-cubes of side length at least $2^{-(k+1)}$, and such
that \emph{over} each $Q_{k+1}$ each point between the \emph{levels}
$m_{Q_{k+1}}$ and $M_{Q_{l}}$ belongs to one of those $(n+1)$-cubes,
is bounded above by\[
\sum_{Q_{k+1}\in\mathcal{Q}_{k+1}}(2^{k+2}\sup_{y\in Q_{k+1}}|h_{k+1}|_{K}(y)|+2).\]
Hence, starting from a covering of $\Gamma(h_{0}|_{K})$ by $(n+1)$-cubes
of side length $2^{-(j_{1}-1)}$ built, with the help of the concept
of oscillation, \emph{over} each $Q_{j_{1}-1}\in\mathcal{Q}_{j_{1}-1}$,
whose number is bounded above by\[
\sum_{Q_{j_{1}-1}\in\mathcal{Q}_{j_{1}-1}}(2^{j_{1}-1}{\rm osc}_{Q_{j_{1}-1}}h_{0}|_{K}+2),\]
and applying Lemma \ref{lem:aggregation} repeatedly (a total number
of $j_{2}-j_{1}$ times), we get the following estimates (the constants
might depend on $t$), where we have also used Remark \ref{rem:dsetcover}
to estimate the number of elements of each $\mathcal{Q}_{j}$:\begin{eqnarray}
\lefteqn{\mathcal{H}_{\sqrt{n+1}\,2^{-j_{1}+1}}^{t}(\Gamma(h_{0}|_{K}+\sum_{j=j_{1}}^{j_{2}-1}h_{j}|_{K}))}\nonumber \\
 & \lesssim & 2^{-j_{1}t}\sum_{Q_{j_{1}-1}\in\mathcal{Q}_{j_{1}-1}}(2^{j_{1}-1}{\rm osc}_{Q_{j_{1}-1}}h_{0}|_{K}+2)\nonumber \\
 &  & +\sum_{j=j_{1}}^{j_{2}-1}2^{-jt}\sum_{Q_{j}\in\mathcal{Q}_{j}}(2^{j+1}\sup_{y\in Q_{j}}|h_{j}|_{K}(y)|+2)\nonumber \\
 & \lesssim & 2^{-j_{1}(t-d)}+2^{-j_{1}(t-1)}\sum_{Q_{j_{1}-1}\in\mathcal{Q}_{j_{1}-1}}{\rm osc}_{Q_{j_{1}-1}}h_{0}|_{K}\label{eq:hausdorff estimate}\\
 &  & +\sum_{j=j_{1}}^{j_{2}-1}2^{-j(t-d)}+\sum_{j=j_{1}}^{j_{2}-1}2^{-j(t-1)}\sum_{m\in\Z^{n}}|\lambda_{jm}|.\nonumber \end{eqnarray}
The estimate in the last term is possible because of the controlled
overlapping between the support of each $\psi_{jm}$ and the different
$Q_{j}$'s (i.e., each ${\rm supp}\,\psi_{jm}$ intersects only a finite
number of $Q_{j}$'s and this number can be bounded above by a constant
independent of $j$). 

Now, if one wants to estimate $\mathcal{H}_{\sqrt{n+1}\,2^{-j_{1}+1}}^{t}(\Gamma(h|_{K}))$
instead, notice that the Lemma \ref{lem:aggregation} can again be
used, and we just have to find out what is the contribution, coming
from $\sum_{j=j_{2}}^{\infty}h_{j}|_{K}$, that we need to add to
the right-hand side of \eqref{eq:hausdorff estimate}. Actually, this
could have been done at the same time we added the contribution of
the term $h_{j_{1}}|_{K}$, resulting in the extra term\begin{eqnarray*}
\lefteqn{2^{-j_{1}(t-1)}\sum_{Q_{j_{1}}\in\mathcal{Q}_{j_{1}}}\sup_{y\in Q_{j_{1}}}|\sum_{j=j_{2}}^{\infty}h_{j}|_{K}(y)|}\\
 & \lesssim & \,2^{-j_{1}(t-1)}2^{j_{1}d}\sup_{y\in\R^{n}}|\sum_{j=j_{2}}^{\infty}h_{j}|_{K}(y)|\,=:\,2^{-j_{1}(t-1)}2^{j_{1}d}C_{j_{2}}.\end{eqnarray*}
However, by the already mentioned uniform convergence of the sum defining
$h$, we have that $C_{j_{2}}$ tends to 0 as $j_{2}$ goes to infinity.
Therefore, by choosing $j_{2}$ large enough (in dependence of $j_{1}$)
so that $C_{j_{2}}2^{j_{1}}\leq1$, the last contribution is just
of the type $2^{-j_{1}(t-d)}$. This is the same as the first term
in \eqref{eq:hausdorff estimate}, and both can be absorbed by the
\emph{third} term in that expression. Hence, for such a choice of
$j_{2}$,\begin{eqnarray}
\lefteqn{\mathcal{H}_{\sqrt{n+1}\,2^{-j_{1}+1}}^{t}(\Gamma(h|_{K}))} &  & .\label{eq:hausdorff est}\\
 & \lesssim & \sum_{j=j_{1}}^{j_{2}-1}2^{-j(t-d)}+2^{-j_{1}(t-1)}\sum_{Q_{j_{1}-1}\in\mathcal{Q}_{j_{1}-1}}{\rm osc}_{Q_{j_{1}-1}}h_{0}|_{K}+\sum_{j=j_{1}}^{j_{2}-1}2^{-j(t-1)}\sum_{m\in\Z^{n}}|\lambda_{jm}|.\nonumber \end{eqnarray}

Now we estimate separately each one of the three distinguished terms
above (in all cases ignoring unimportant multiplicative constants):

The first one is dominated by $2^{-j_{1}(t-d)}$ under the assumption
$t>d$.

The second one is dominated by \begin{eqnarray}
\lefteqn{2^{-j_{1}(t-1)}\sum_{Q_{j_{1}-1}\in\mathcal{Q}_{j_{1}-1}}\sum_{j=1}^{j_{1}-1}\sum_{m}{\rm osc}_{Q_{j_{1}-1}}(\lambda_{jm}\psi_{jm}|_{K})}\nonumber \\
 & \leq & 2^{-j_{1}(t-1)}\sum_{j=1}^{j_{1}-1}\sum_{m}\sum_{Q_{j_{1}-1}\in\mathcal{Q}_{j_{1}-1}}|\lambda_{jm}|\,|\nabla\psi_{jm}(\xi_{j_{1}})|\,2^{-j_{1}},\label{eq:2nd term}\end{eqnarray}
where $\xi_{j_{1}}$ is chosen in $Q_{j_{1}-1}$ in accordance with
the mean value theorem (which we have just used above) and in $\sum_{m}$
the $m$ is restricted to the values for which ${\rm supp}\,\psi_{jm}$
intersects $K$. Such a number of $m$'s can clearly be estimated
from above by $2^{jd}$ (cf. Remark \ref{rem:dsetcover} and Definition
\ref{def:wavelets}). Next we remark that $|\nabla\psi_{jm}(\xi_{j_{1}})|\lesssim2^{j}$
and that in the inner sum in \eqref{eq:2nd term} we only need to
consider the $Q_{j_{1}-1}$'s which intersect ${\rm supp}\,\psi_{jm}$.
It is not difficult to see that such a number of $Q_{j_{1}-1}$'s
can be estimated from above by $2^{(j_{1}-j)d}$. Putting all this
together, and using also the estimate \eqref{eq:est sum lambda} and
the hypothesis $0<s<1$, the second term on the right-hand side of
\eqref{eq:hausdorff est} is dominated by\begin{eqnarray*}
\lefteqn{2^{-j_{1}(t-1)}\sum_{j=1}^{j_{1}-1}2^{j-j_{1}}2^{(j_{1}-j)d}\sum_{m\in\Z^{n}}|\lambda_{jm}|}\\
 & \lesssim & 2^{-j_{1}(t-d)}j_{1}\max_{1\leq j\leq j_{1}-1}2^{-j(s-1)}\:\approx\: j_{1}2^{-j_{1}(t-d+s-1)}.\end{eqnarray*}

Finally, again by using the estimate \eqref{eq:est sum lambda}, the
third term on the right-hand side of \eqref{eq:hausdorff est} can
be dominated by $2^{-j_{1}(t-d+s-1)}$ under the assumption $t>d-s+1$.

Altogether, and under the assumption $t>d-s+1$ (which, in particular,
implies that $t>d$, due to the hypothesis $0<s<1$), we have obtained
that\[
\mathcal{H}_{\sqrt{n+1}\,2^{-j_{1}+1}}^{t}(\Gamma(h|_{K}))\lesssim j_{1}2^{-j_{1}(t-d+s-1)},\]
from which it follows that\[
0\leq\mathcal{H}^{t}(\Gamma(h|_{K}))=\lim_{\delta\to0+}\mathcal{H}_{\delta}^{t}(\Gamma(h|_{K}))=\lim_{j_{1}\to\infty}\mathcal{H}_{\sqrt{n+1}\,2^{-j_{1}+1}}^{t}(\Gamma(h|_{K}))\leq0,\]
that is, $\mathcal{H}^{t}(\Gamma(h|_{K}))=0$. 

This being true for any $t>d+1-s$, we obtain, by definition, that
$\dim_{H}(\Gamma(h|_{K}))\leq d+1-s$, that is, $\dim_{H}(\Gamma(h|_{K}))\leq H(d,s)$
in the case $s\leq d$.

\medskip

We assume now that $s>d$ and show that also $\dim_{H}(\Gamma(h|_{K}))\leq H(d,s)$,
which means $\dim_{H}(\Gamma(h|_{K}))\leq d/s$ in this case (cf.
definition of $H(d,s)$ in \eqref{eq:H(d,s)}). As we shall see, it
will be enough to estimate $\overline{\dim}_{B}(\Gamma(h|_{K}))$
and use the relation $\dim_{H}(\Gamma(h|_{K}))\leq\overline{\dim}_{B}(\Gamma(h|_{K}))$.

Given $\nu\in\N$, we start by covering $K$ by $\approx2^{\nu d/s}$
$n$-cubes $I_{\nu}$ of side length $2^{-\nu/s}$ taken from a given
regular tessellation of $\Rn$ by cubes of sides parallel to the axes
and of such side length (we know, from Remark \ref{rem:dsetcover},
that this is possible). Using the \emph{representation} \eqref{eq:h}
for $h$, we can write\begin{eqnarray*}
\sum_{I_{\nu}}{\rm osc}_{I_{\nu}}h|_{K} & \leq & \sum_{I_{\nu}}\sum_{j=1}^{\infty}\sum_{m\in\Z^{n}}{\rm osc}_{I_{\nu}}(\lambda_{jm}\psi_{jm}|_{K})\\
 & = & \sum_{j=1}^{j_{\nu}-1}\sum_{m\in\Z^{n}}\sum_{I_{\nu}}{\rm osc}_{I_{\nu}}(\lambda_{jm}\psi_{jm}|_{K})+\sum_{j=j_{\nu}}^{\infty}\sum_{m\in\Z^{n}}\sum_{I_{\nu}}{\rm osc}_{I_{\nu}}(\lambda_{jm}\psi_{jm}|_{K})\\
 & =: & (I)+(II),\end{eqnarray*}
 where $j_{\nu}\in\N$ was chosen in such a way that $\nu/s\leq j_{\nu}\leq\nu/s+1$.
In particular, $2^{-j_{\nu}}\approx2^{-\nu/s}$. Reasoning now as
was done to control the second term on the right-hand side of \eqref{eq:hausdorff est},
we can dominate (I) by\begin{eqnarray*}
\lefteqn{\sum_{j=1}^{j_{\nu}-1}\sum_{m}\sum_{I_{\nu}}|\lambda_{jm}|\,|\nabla\psi_{jm}(\xi_{\nu})|\,2^{-j_{\nu}}}\\
 & \lesssim & \sum_{j=1}^{j_{\nu}-1}2^{j-j_{\nu}}2^{(j_{\nu}-j)d}\sum_{m\in\Z^{n}}|\lambda_{jm}|\:\lesssim\: j_{\nu}2^{-j_{\nu}(s-d)}\:\approx\:\nu2^{-\nu(s-d)/s}.\end{eqnarray*}
On the other hand, (II) can be dominated by \begin{eqnarray*}
\lefteqn{\sum_{j=j_{\nu}}^{\infty}\sum_{m\in\Z^{n}}\sum_{I_{\nu}}\sup_{I_{\nu}}|\lambda_{jm}\psi_{jm}|_{K}|}\\
 & \lesssim & \sum_{j=j_{\nu}}^{\infty}\sum_{m\in\Z^{n}}\sup_{\R^{n}}|\lambda_{jm}\psi_{jm}|_{K}|\:\lesssim\:\sum_{j=j_{\nu}}^{\infty}2^{-j(s-d)}\:\approx\:2^{-\nu(s-d)/s}.\end{eqnarray*}
The first estimate above is possible because of the controlled overlapping
between the support of each $\psi_{jm}$ and the different $I_{\nu}$'s
(i.e., each ${\rm supp}\,\psi_{jm}$ intersects only a finite number
of $I_{\nu}$'s and this number can be bounded above by a constant
independent of $j$ and $j_{\nu}$, due to the fact that here we have
$j\geq j_{\nu}$). The second estimate follows from \eqref{eq:est sum lambda}.
Summing up,\[
\sum_{I_{\nu}}{\rm osc}_{I_{\nu}}h|_{K}\lesssim\nu2^{-\nu(s-d)/s}.\]

Now consider a covering of the graph of $h|_{K}$ by $(n+1)$-cubes
of side length $2^{-\nu}$ taken from a corresponding regular tessellation
of $\R^{n+1}$ by dyadic cubes of sides parallel to the axes. Recalling
the assumption $0<s<1$, it is clear that the number of such cubes
does not exceed $\sum_{I_{\nu}}(2^{\nu}{\rm osc}_{I_{\nu}}h|_{K}+1)$.
Then, from the estimate above we see that this number is dominated
by\[
\sum_{I_{\nu}}1+2^{\nu}\sum_{I_{\nu}}{\rm osc}_{I_{\nu}}h|_{K}\:\lesssim\:2^{\nu d/s}+2^{\nu}\nu2^{-\nu(s-d)/s}\:\approx\:\nu2^{\nu d/s}.\]
Therefore, \[
\overline{\dim}_{B}(\Gamma(h|_{K}))\leq\lim_{\nu\to\infty}\left(\frac{\log_{2}\nu}{\nu}+\frac{\nu d/s}{\nu}\right)=\frac{d}{s}.\]

We recall that we have been assuming $0<s<1$. We deal now with the
case $s=1$.

Given a real continuous function $f$ in $\mathbb{B}_{pq}^{1}(K)$,
then we also have $f\in\mathbb{B}_{pq}^{s}(K)$ for any $s\in(0,1)$,
therefore $\dim_{H}\Gamma(f)\leq H(d,s)$ for any such $s$. Hence,
if $d\geq1$, also $d>s$ and $\dim_{H}\Gamma(f)\leq d+1-s$; letting
$s\to1^{-}$, we get $\dim_{H}\Gamma(f)\leq d=H(d,1)$. If $d<1$,
choose any $s\in[d,1)$, so that $s\geq d$ and, therefore, $\dim_{H}\Gamma(f)\leq d/s$;
again letting $s\to1^{-}$, it follows $\dim_{H}\Gamma(f)\leq d=H(d,1)$.\end{proof}
\begin{cor}
\label{cor:Corollary}Consider $0<d\leq n$ and $K$ a $d$-set. Let
$1\leq p\leq\infty$, $0<q\leq\infty$ and $0<s\leq1$. Then the estimates
of Proposition \ref{pro:lessthan} and Theorem \ref{thm:Besov up est}
are sharp. That is, $\sup_{f}\dim_{H}\Gamma(f)=H(d,s)$, where the
supremum is taken over all real continuous functions belonging either
to $\mathcal{C}^{s}(K)$ or to $\mathbb{B}_{pq}^{s}(K)$.\end{cor}
\begin{proof}
Consider first the case of the spaces $\mathcal{C}^{s}(K)$. If $0<s<1$,
it follows from Example \ref{exa:holder} and Theorem \ref{thm:weierstrass}
that there exists a real continuous function in $\mathcal{C}^{s}(K)$
--- namely $W_{s,\theta}|_{K}$, for some $\theta$ --- the graph
of which has Hausdorff dimension exactly equal to $H(d,s)$. If $s=1$,
then $H(d,s)=H(d,1)=d$ and the result follows from Lemma \ref{lem:minmax}.

As to the spaces $\mathbb{B}_{pq}^{s}(K)$, for $s=1$ it follows
exactly as just pointed out, so we only need to consider $s\in(0,1)$: 

Let $\varepsilon\in(0,1-s)$, so that $0<s+\varepsilon<1$. Recall
--- see Example \ref{exa:holder} --- that $W_{s+\varepsilon,\theta}$
is H\"older continuous of exponent $s+\varepsilon$ on any bounded
subset of $\R^{n}$. Considering then an open bounded set $V\supset K$
and a function $\psi\in\mathcal{C}^{1}(\R^{n})$ with $\psi\equiv1$
on $K$ and $\psi\equiv0$ outside $V$, we have $\psi W_{s+\varepsilon,\theta}\in\mathcal{C}^{s+\varepsilon}(\R^{n})$.
Since the latter space is contained in $B_{\infty\infty}^{s+\varepsilon}(\R^{n})$
--- cf. \cite[pp. 4, 5 and 17]{Tri92} ---, and this, in turn, is
embedded in $B_{\infty q}^{s}(\R^{n})$, then, using Remark \ref{rem:continuoustrace},
Definition \ref{def:Besov on K} and Proposition \ref{pro:emb p2<p1},
we can write \[
f_{\varepsilon}:=W_{s+\varepsilon,\theta}|_{K}=(\psi W_{s+\varepsilon,\theta})|_{K}=tr_{K}(\psi W_{s+\varepsilon,\theta})\in\mathbb{B}_{\infty q}^{s}(K)\subset\mathbb{B}_{pq}^{s}(K).\]

Therefore, given any $\varepsilon\in(0,1-s)$ there exists a real
continuous function $f_{\varepsilon}\in\mathbb{B}_{pq}^{s}(K)$
the graph of which has Hausdorff dimension equal to $H(d,s+\varepsilon)$.
If $s<d$, restrict further the $\varepsilon$ to be also in $(0,d-s)$,
so that $H(d,s+\varepsilon)=d+1-s-\varepsilon$, hence $\sup_{f_{\varepsilon}}\dim_{H}\Gamma(f)=d+1-s=H(d,s)$.
If $s\geq d$, then also $s+\varepsilon\geq d$, so that $H(d,s+\varepsilon)=d/(s+\varepsilon)$,
hence $\sup_{f_{\varepsilon}}\dim_{H}\Gamma(f)=d/s=H(d,s)$ too. 
\end{proof}
\bibliographystyle{plain}
\bibliography{bibliography}

\end{document}